\documentclass{article}

\usepackage{amsmath,amssymb,enumerate,comment}
\usepackage{amsthm,cancel}
\usepackage[usenames,dvipsnames]{color}
\usepackage{natbib}

\usepackage[margin=1.25in]{geometry}

\usepackage{graphicx} 
\usepackage{subfigure} 
\usepackage{multirow}
\usepackage{xcolor}
\usepackage{mathtools}
\usepackage{relsize}
\usepackage{url}
\usepackage{nicefrac}
\usepackage{xspace}
\usepackage{algorithm}
\usepackage{algorithmic}

\newtheorem{lemma}{Lemma}
\newtheorem{theorem}{Theorem}
\newtheorem{corollary}{Corollary}

\newtheorem{definition}{Definition}

\newtheorem*{lemma*}{Lemma}
\newtheorem*{theorem*}{Theorem}
\newtheorem*{corollary*}{Corollary}
\newtheorem*{remark*}{Remark}
\newtheorem*{definition*}{Definition}

\newcommand{\sag}{\textsc{Sag}}
\newcommand{\sgd}{\textsc{Sgd}\xspace}

\newcommand{\saga}{\textsc{Saga}\xspace}
\newcommand{\regsaga}{\textsc{Reg-Saga}}

\newcommand{\gd}{\textsc{GradDescent}\xspace}
\newcommand{\sml}[1]{{\small #1}}
\newcommand{\fromto}[3]{\sml{$#1{\;\le\;}#2{\;\le\;}#3$}}

\newcommand{\reals}{\mathbb{R}}

\newcommand{\nlsum}{\sum\nolimits}
\newcommand{\E}{\mathbb{E}}
\newcommand{\Fc}{\mathcal{F}}

\renewcommand{\cite}[1]{\citep{#1}}
\setcitestyle{authoryear,round,citesep={;},aysep={,},yysep={;}}

\begin{document}

\title{Fast Incremental Method for Nonconvex Optimization}

\author{\\
Sashank J. Reddi\\
\texttt{sjakkamr@cs.cmu.edu}\\
Carnegie Mellon University \\
\and \\
Suvrit Sra\\
\texttt{suvrit@mit.edu} \\
Massachusetts Institute of Technology \\
\and \\
Barnab\'{a}s P\'{o}cz\'os\\
\texttt{bapoczos@cs.cmu.edu} \\
Carnegie Mellon University \\
\and \\
Alex Smola \\
\texttt{alex@smola.org} \\ 
Carnegie Mellon University \\
}

%

%
%
%
%
%
\date{}
\maketitle

\begin{abstract}
  We analyze a fast incremental aggregated gradient method for optimizing nonconvex problems of the form $\min_x \sum_i f_i(x)$. Specifically, we analyze the \saga algorithm within an Incremental First-order Oracle framework, and show that it converges to a stationary point provably faster than both gradient descent and stochastic gradient descent. We also discuss a Polyak's special class of nonconvex problems for which \saga converges at a linear rate to the global optimum. Finally, we analyze the practically valuable regularized and minibatch variants of \saga. To our knowledge, this paper presents the first analysis of fast convergence for an incremental aggregated gradient method for nonconvex problems.

\end{abstract}

\section{Introduction}
We study the \emph{finite-sum} optimization problem
\begin{equation}
  \label{eq:1}
  \min_{x\in \reals^d}\ f(x) := \frac{1}{n}\sum_{i=1}^n f_i(x),
\end{equation}
where each $f_i$ ($i \in \{1,\dots, n\} \triangleq [n]$) can be nonconvex; our only assumption is that the gradients of $f_i$ exist and are Lipschitz continuous. We denote the class of such instances of~\eqref{eq:1} by $\Fc_n$.

Problems of the form~\eqref{eq:1} are of central importance in machine learning where they occur as instances of empirical risk minimization; well-known examples include logistic regression (convex)~\cite{logreg} and deep neural networks (nonconvex)~\cite{dnn}. Consequently, such problems have been intensively studied. A basic approach for solving~\eqref{eq:1} is gradient descent ($\gd$), described by the following update rule:
\begin{align}
\label{eq:gd}
x^{t+1} = x^t - \eta_t \nabla f(x^{t}), \text{ where } \eta_t > 0.
\end{align}
However, iteration~\eqref{eq:gd} is prohibitively expensive in large-scale settings where $n$ is very large. For such settings, stochastic and incremental methods are  typical~\cite{bertsekas11.survey,Defazio14}. These methods use cheap \emph{noisy} estimates of the gradient at each iteration of~\eqref{eq:gd} instead of $\nabla f(x^t)$. A particularly popular approach, stochastic gradient descent (\sgd) uses $\nabla f_{i_t}$, where $i_t$ in chosen uniformly randomly from $\{1, \dots, n\}$. While the cost of each iteration is now greatly reduced, it is not without any drawbacks. Due to the \emph{noise} (also known as variance in stochastic methods) in the gradients, one has to typically use decreasing step-sizes $\eta_t$ to ensure convergence, and consequently, the convergence rate gets adversely affected.

Therefore, it is of great interest to overcome the slowdown of \sgd without giving up its scalability. Towards this end, for convex instances of~\eqref{eq:1}, remarkable progress has been made recently. The key realization is that if we make multiple passes through the data, we can store information that allows us to reduce variance of the stochastic gradients. As a result, we can use constant stepsizes (rather than diminishing scalars) and obtain convergence faster than \sgd, both in theory and practice~\cite{Johnson13,Schmidt13,Defazio14}.

Nonconvex instances of~\eqref{eq:1} are also known to enjoy similar speedups~\cite{Johnson13}, but existing analysis does not explain this success as it relies heavily on convexity to control variance. Since \sgd dominates large-scale nonconvex optimization (including neural network training), it is of great value to develop and analyze faster nonconvex stochastic methods.

With the above motivation, we analyze \saga \cite{Defazio14}, an incremental aggregated gradient algorithm that extends the seminal \sag{} method of~\cite{Schmidt13}, and has been shown to work well for convex finite-sum problems~\cite{Defazio14}. Specifically, we analyze \saga for the class $\Fc_n$ using the \emph{incremental first-order oracle} (IFO) framework~\cite{agarwal2014}. For $f \in \mathcal{F}_n$, an IFO is a subroutine that takes an index $i \in [n]$ and a point $x \in \mathbb{R}^d$, and returns the pair $(f_i(x),\nabla f_i(x))$. 

To our knowledge, this paper presents the first analysis of fast convergence for an incremental aggregated gradient method for \emph{nonconvex} problems. The attained rates improve over both \sgd and \gd, a benefit that is also corroborated by experiments. Before presenting our new analysis, let us briefly recall some items of related work.

\subsection{Related work}
A concise survey of incremental gradient methods is~\cite{bertsekas11.survey}. An accessible analysis of stochastic convex optimization ($\min \E_z[F(x,z)]$) is~\cite{nemirov09}. Classically, \sgd stems from the seminal work \cite{RobMon51}, and has since witnessed many developments \cite{kushner2012}, including parallel and distributed variants \cite{BerTsi89,agDuc11,Recht11}, though non-asymptotic convergence analysis is limited to convex setups. Faster rates for convex problems in $\mathcal{F}_n$ are attained by variance reduced stochastic methods, e.g., \cite{Defazio14,Johnson13,Schmidt13,sdca,Reddi2015}. Linear convergence of stochastic dual coordinate ascent when $f_i$ ($i\in[n]$) may be nonconvex but $f$ is strongly convex is studied in \cite{Shwartz15}. Lower bounds for convex finite-sum problems are studied in \cite{agarwal2014}.

For nonconvex nonsmooth problems the first incremental proximal-splitting methods is in \cite{Sra2012}, though only asymptotic convergence is studied. Hong \cite{hong2014} studies convergence of a distributed nonconvex incremental ADMM algorithm. The first work to present non-asymptotic convergence rates for \sgd is \cite{Ghadimi13}; this work presents an $O(1/\epsilon^2)$ iteration bound for \sgd to satisfy approximate stationarity $\|\nabla f(x)\|^2 \le \epsilon$, and their convergence criterion is motivated by the gradient descent analysis of Nesterov~\cite{nesterov03}. The first analysis for nonconvex variance reduced stochastic gradient is due to~\cite{shamir2014stochastic}, who apply it to the specific problem of principal component analysis (PCA).

\section{Preliminaries}
In this section, we further explain our assumptions and goals. We say $f$ is $L$-\emph{smooth} if there is a constant $L$ such that
\begin{equation*}
  \|\nabla f(x)-\nabla f(y)\| \le L\|x-y\|,\quad\forall\ x, y \in \reals^d.
\end{equation*}
Throughout, we assume that all $f_i$ in~\eqref{eq:1} are $L$-smooth, i.e., $\|\nabla f_i(x)-\nabla f_i(y)\| \le L\|x-y\|$ for all $i \in [n]$. Such an assumption is common in the analysis of first-order methods. For ease of exposition, we assume the Lipschitz constant $L$ to be independent of $n$. For our analysis, we also discuss the class of $\tau$-\emph{gradient dominated} \cite{Polyak1963,nesterov2006} functions, namely functions for which
\begin{equation}
  \label{eq:2}
  f(x) - f(x^*) \leq \tau \|\nabla f(x)\|^2,
\end{equation}
where $x^*$ is a global minimizer of $f$. This class of functions was originally introduced by Polyak in \cite{Polyak1963}. Observe that such functions need not be convex. Also notice that gradient dominance \eqref{eq:2} is a restriction on the overall function $f$, but not on the individual functions $f_i$ ($i\in [n]$).

Following \cite{nesterov03,Ghadimi13} we use $\|\nabla f(x)\|^2 \le \epsilon$ to judge approximate stationarity of $x$. Contrast this with \sgd for convex $f$, where one uses $[f(x) - f(x^*)]$ or $\|x - x^*\|^2$ as criteria for convergence analysis. Such criteria cannot be used for nonconvex functions due to the intractability of the problem. 

While the quantities $\|\nabla f(x)\|^2$, $[f(x) - f(x^*)]$, or $\|x - x^*\|^2$ are not comparable in general, they are typically assumed to be of similar magnitude (see \cite{Ghadimi13}). We note that our analysis does \emph{not} assume $n$ to be a constant, so we report dependence on it in our results. Furthermore, while stating our complexity results, we assume that the initial point of the algorithms is constant, i.e., $f(x^0) - f(x^*)$ and $\|x^0 - x^*\|$ are constants. For our analysis, we will need the following definition.

\begin{definition}
\label{def:eps-accurate}
A point $x$ is called $\epsilon$-accurate if $\|\nabla f(x)\|^2 \leq \epsilon$. A stochastic iterative algorithm is said to achieve $\epsilon$-accuracy in $t$ iterations if $\mathbb{E}[\|\nabla f(x^t)\|^2] \leq \epsilon$, where the expectation is taken over its stochasticity.
\end{definition}


We start our discussion of algorithms by recalling \sgd, which performs the following update in its $t^{\text{th}}$ iteration:
\begin{align}
\label{eq:sgd}
x^{t+1} = x^{t} - \eta_t \nabla f_{i_t}(x),
\end{align}
where $i_t$ is a random index chosen from $[n]$, and the gradient $\nabla f_{i_t}(x^t)$ approximates the gradient of $f$ at $x^t$, $\nabla f(x^t)$. It can be seen that the update is unbiased, i.e., $\mathbb{E}[\nabla f_{i_t}(x^t)] = \nabla f(x^t)$ since the index $i_t$ is chosen uniformly at random. Though the \sgd update is unbiased, it suffers from variance due to the aforementioned stochasticity in the chosen index. To control the variance one has to decrease the step size $\eta_t$ in~\eqref{eq:sgd}, which in turn leads to slow convergence. The following is a well-known result on $\sgd$ in the context of nonconvex optimization \cite{Ghadimi13}.

\begin{theorem}
  \label{thm:sgd-oracle}
  Suppose $\|\nabla f_i\| \leq \sigma$ i.e., gradient of function $f_i$ is bounded for all $i \in [n]$, then the IFO complexity of $\sgd$ to obtain an $\epsilon$-accurate solution is $O(1/\epsilon^2)$.
\end{theorem}

It is instructive to compare the above result with the convergence rate of $\gd$. The IFO complexity of $\gd$ is $O(n/\epsilon)$. Thus, while $\sgd$ eliminates the dependence on $n$, the convergence rate worsens to $O(1/\epsilon^2)$ from $O(1/\epsilon)$ in $\gd$. In the next section, we investigate an incremental method with faster convergence.

\section{Algorithm}
We describe below the \saga algorithm and prove its fast convergence for \emph{nonconvex} optimization. $\saga$ is a popular incremental method in machine learning and optimization communities. It is very effective in reducing the variance introduced due to stochasticity in $\sgd$.
Algorithm~\ref{alg:saga} presents pseudocode for \saga. Note that the update $v^t$ (Line 5) is unbiased, i.e., $\mathbb{E}[v^t] = \nabla f(x^t)$. This is due to the uniform random selection of index $i_t$. It can be seen in Algorithm~\ref{alg:saga} that $\saga$ maintains gradients at $\alpha_{i}$ for $i \in [n]$. This additional set is critical to reducing the variance of the update $v^t$. At each iteration of the algorithm, one of the $\alpha_i$ is updated to the current iterate. An implementation of $\saga$ requires storage and updating of gradients $\nabla f_i(\alpha_i)$; the storage cost of the algorithm is $nd$. While this leads to higher storage in comparison to $\sgd$, this cost is often reasonable for many applications. Furthermore, this cost can be reduced in the case of specific models; refer to \cite{Defazio14} for more details. 

For ease of exposition, we introduce the following quantity:
\begin{align}
\label{eq:Gamma-t}
\Gamma_{t} = \bigl(\eta - \tfrac{c_{t+1}\eta}{\beta} - \eta^2L - 2c_{t+1}\eta^2\bigr),
\end{align}
where the parameters $c_{t+1}$, $\beta$ and $\eta$ will be defined shortly. We start with the following set of key results that provide convergence rate of Algorithm~\ref{alg:saga}.

\begin{algorithm}[tb]\small
   \caption{SAGA$\left(x^0,T, \eta\right)$}
   \label{alg:saga}
\begin{algorithmic}[1]
   \STATE {\bfseries Input:} $x^0 \in \mathbb{R}^d$, $\alpha_{i}^0 = x^0$ for $i \in [n]$,  number of iterations $T$, step size $\eta > 0$
   \STATE $g^{0} = \frac{1}{n} \sum_{i=1}^n \nabla f_{i}(\alpha_i^{0})$
   \FOR{$t=0$ {\bfseries to} $T-1$}
   \STATE Uniformly randomly pick $i_t, j_t$ from $[n]$ 
   \STATE $v^t =  \nabla f_{i_t}(x^t) - \nabla f_{i_t}(\alpha_{i_t}^{t}) + g^{t}$
   \STATE $x^{t+1} = x^{t} - \eta v^t$
   \STATE $\alpha_{j_t}^{t+1} = x^t$ and $\alpha_{j}^{t+1} = \alpha_{j}^{t}$ for $j \neq j_t$
   \STATE $g^{t+1} = g^t - \frac{1}{n}(\nabla f_{j_t}(\alpha_{j_t}^t) - \nabla f_{j_t}(\alpha_{j_t}^{t+1}))$
   \ENDFOR
   \STATE {\bfseries Output:} Iterate $x_a$ chosen uniformly random from $\{x^t\}_{t=0}^{T-1}$.
\end{algorithmic}
\end{algorithm}

\begin{lemma}
\label{lem:nonconvex-saga}
For $c_t, c_{t+1}, \beta > 0$, suppose we have 
$$
c_{t} = c_{t+1}(1 - \tfrac{1}{n} +  \eta\beta + 2\eta^2L^2 ) +  \eta^2L^3.
$$ 
Also let $\eta$, $\beta$ and $c_{t+1}$ be chosen such that $\Gamma_{t} > 0$. Then, the iterates $\{x^t\}$ of Algorithm~\ref{alg:saga} satisfy the bound
\begin{align*}
\mathbb{E}[\|\nabla f(x^{t})\|^2] \leq \frac{R^{t} - R^{t+1}}{\Gamma_t},
\end{align*}
where $R^{t} := \mathbb{E}[f(x^{t}) +  (c_{t}/n) \sum_{i=1}^n \|x^{t} - \alpha_i^{t}\|^2]$.
\end{lemma}
The proof of this lemma is given in Section~\ref{sec:proofs}. Using this lemma we prove the following result on the iterates of $\saga$.

\begin{theorem}
\label{thm:nonconvex-inter}
  Let $f \in \Fc_n$. Let $c_T = 0$, $\beta > 0$, and $c_{t} = c_{t+1}(1 - \tfrac{1}{n} + \eta\beta + 2\eta^2L^2) +  \eta^2L^3$ be such that $\Gamma_t > 0$ for \fromto{0}{t}{T-1}. Define the quantity $\gamma_n := \min_{0 \leq t \leq T-1} \Gamma_t$. Then  the output $x_a$ of Algorithm~\ref{alg:saga} satisfies the bound
  \begin{align*}
    \mathbb{E}[\|\nabla f(x_a)\|^2] \leq \frac{f(x^{0}) - f(x^*)}{T\gamma_n},
  \end{align*}
  where $x^*$ is an optimal solution to~\eqref{eq:1}.
\end{theorem}
\begin{proof}
  We apply telescoping sums to the result of Lemma~\ref{lem:nonconvex-saga} to obtain
\begin{align*}
   \gamma_n \nlsum_{t=0}^{T-1} \mathbb{E}[\|\nabla f(x^{t})\|^2] &\leq \nlsum_{t=0}^{T-1} \Gamma_t \mathbb{E}[\|\nabla f(x^{t})\|^2] \\
   &\leq R^{0} - R^{T}.
\end{align*}
The first inequality follows from the definition of $\gamma_n$. This inequality in turn implies the bound
\begin{align}
  \label{eq:descent-property}
  \nlsum_{t=0}^{T-1} \mathbb{E}[\|\nabla f(x^{t})\|^2] \leq \frac{\mathbb{E}[f(x^0) - f(x^{T})]}{\gamma_n},
\end{align}
where we used that $R^{T} = \mathbb{E}[f(x^{T})]$ (since $c_T = 0$), and that $R^{0} = \mathbb{E}[f(x^0)]$ (since $\alpha_i^0 = x^0$ for $i \in [n]$). Using inequality~\eqref{eq:descent-property}, the optimality of $x^*$, and the definition of $x_a$ in Algorithm~\ref{alg:saga}, we obtain the desired result.
\end{proof}

Note that the notation $\gamma_n$ involves $n$, since this quantity can depend on $n$. To obtain an explicit dependence on $n$, we have to use an appropriate choice of $\beta$ and $\eta$. This is made precise by the following main result of the paper.

\begin{theorem}
\label{thm:nonconvex-gen}
Suppose $f \in \Fc_n$. Let $\eta = 1/(3Ln^{2/3})$ and $\beta = L/n^{1/3}$. Then,  $\gamma_n \geq \frac{1}{12Ln^{2/3}}$ and we have the bound
\begin{align*}
\mathbb{E}[\|\nabla f(x_a)\|^2] &\leq \frac{12Ln^{2/3} [f(x^{0}) - f(x^*)]}{T},
\end{align*} 
where $x^*$ is an optimal solution to the problem in~\eqref{eq:1} and $x_a$ is the output of Algorithm~\ref{alg:saga}.
\end{theorem}
\begin{proof}
With the values of $\eta$ and $\beta$, let us first establish an upper bound on $c_t$. Let $\theta$ denote $\tfrac{1}{n} - \eta\beta - 2\eta^2L^2$. Observe that $\theta < 1$ and $\theta \geq 4/(9n)$. This is due to the specific values of $\eta$ and $\beta$ stated in the theorem. Also, we have $c_{t} = c_{t+1}(1 - \theta) +  \eta^2L^3$. Using this relationship, it is easy to see that $c_{t} = \eta^2L^3 \tfrac{1 - (1 - \theta)^{T-t}}{\theta}$. Therefore, we obtain the bound
\begin{align}
\label{eq:c-t-upperbound}
c_{t}  = \eta^2L^3 \tfrac{1 - (1 - \theta)^{T-t}}{\theta} \leq \frac{\eta^2L^3}{\theta} \leq \frac{L}{4n^{1/3}},
\end{align}
for all $0 \leq t \leq T$, where the inequality follows from the definition of $\eta$ and the fact that $\theta \geq 4/(9n)$. Using the above upper bound on $c_t$ we can conclude that
\begin{align*}
\gamma_n &= \min_t \bigl(\eta - \tfrac{c_{t+1}\eta}{\beta} - \eta^2L - 2c_{t+1}\eta^2\bigr) \geq \frac{1}{12Ln^{2/3}},
\end{align*}
upon using the following inequalities: (i) $c_{t+1}\eta/\beta \leq \eta/4$, (ii) $\eta^2L \leq \eta/3$ and (iii) $2c_{t+1}\eta^2 \leq \eta/6$, which hold due to the upper bound on $c_t$ in~\eqref{eq:c-t-upperbound}. Substituting this bound on $\gamma_n$ in Theorem~\ref{thm:nonconvex-inter}, we obtain the desired result.
\end{proof}

A more general result with step size $\eta < 1/(3Ln^{2/3})$ can be proved, but it will only lead to a theoretically suboptimal convergence result. Recall that each iteration of Algorithm~\ref{alg:saga} requires $O(1)$ IFO calls. Using this fact, we can rewrite Theorem~\ref{thm:nonconvex-gen} in terms of its IFO complexity as follows.
\begin{corollary}
\label{cor:saga-nonconvex-oracle}
If $f \in \Fc_n$, then the IFO complexity of Algorithm~\ref{alg:saga} (with parameters from Theorem~\ref{thm:nonconvex-gen}) to obtain an $\epsilon$-accurate solution is $O(n + n^{2/3}/\epsilon)$.
\end{corollary}
This corollary follows from the $O(1)$ per iteration cost of Algorithm~\ref{alg:saga} and because $n$ IFO calls are required to calculate $g^0$ at the start of the algorithm. In special cases, the initial $O(n)$ IFO calls can be avoided (refer to \cite{Schmidt13,Defazio14} for details). By comparing the IFO complexity of $\saga$ ($O(n + n^{2/3}/\epsilon)$) with that of $\gd$ ($O(n/\epsilon)$), we see that $\saga$ is faster than $\gd$ by a factor of $n^{1/3}$.

\section{Finite Sums with Regularization}
\label{sec:reg}
In this section, we study the problem of finite-sum problems with additional regularization. More specifically, we consider problems of the form
\begin{equation}
  \label{eq:reg-1}
  \min_{x\in \reals^d}\ f(x) := \frac{1}{n}\sum_{i=1}^n f_i(x) + r(x),
\end{equation}
where $r:\mathbb{R}^d \rightarrow \mathbb{R}$ is an $L$-smooth (possibly nonconvex) function. Problems of this nature arise in machine learning where the functions $f_i$ are loss functions and $r$ is a regularizer. Since we assumed $r$ to be smooth, \eqref{eq:reg-1} can be reformulated as \eqref{eq:1} by simply encapsulating $r$ into the functions $f_i$. However, as we will see, it is  beneficial to handle the regularization separately. We call the variant of $\saga$ with the additional regularization as $\regsaga$. The key difference between $\regsaga$ and $\saga$ lies in Line 6 of Algorithm~\ref{alg:saga}. In particular, for $\regsaga$, Line 6 of Algorithm~\ref{alg:saga} is replaced with the following update:
\begin{align}
\label{eq:reg-saga-update}
x^{t+1} = x^{t} - \eta(v^t + \nabla r(x^t)).
\end{align}
Note that the primary difference is that the part of gradient based on function $r$ is updated at each iteration of the algorithm. The convergence of $\regsaga$ can be proved along the lines of our analysis of $\saga$. Hence, we omit the details for brevity and directly state the following key result stating the IFO complexity of $\regsaga$.
\begin{theorem}
\label{thm:reg-saga-nonconvex-oracle}
If function $f$ is of the form in~\eqref{eq:reg-1}, then the IFO complexity of $\regsaga$ to obtain an $\epsilon$-accurate solution is $O(n + n^{2/3}/\epsilon)$.
\end{theorem}

The proof essentially follows along the lines of the proof of Theorem~\ref{thm:nonconvex-gen}. The difference, however, being that the update corresponding to function $r(x)$ is handled explicitly at each iteration. Note that the above IFO complexity is not different from that in Corollary~\ref{cor:saga-nonconvex-oracle}. However, its main benefit comes in the form of storage efficiency in problems with more structure. To understand this, consider the problems of form
\begin{equation}
  \label{eq:reg-loss}
  \min_{x\in \reals^d}\ f(x) := \frac{1}{n}\sum_{i=1}^n l(x^{\top}z_i) + r(x),
\end{equation}
where $z_i \in \mathbb{R}^d$ for $i \in [n]$ while $l :\mathbb{R} \rightarrow \mathbb{R}_{\ge 0}$ is a differentiable loss function. Here, $l$ and $r$ can be in general nonconvex. Such problems are popularly known as (regularized) empirical risk minimization in machine learning literature. We can directly apply $\saga$ to \eqref{eq:reg-loss} by casting it in the form \eqref{eq:1}. However, recall that the storage cost of $\saga$ is $O(nd)$ due to the cost of storing the gradient at $\alpha_i^t$. This storage cost can be avoided in $\regsaga$ by handling the function $r$ separately. Indeed, for $\regsaga$ we need to store just $\nabla l(x^\top z_i)$ for all $i \in [n]$ (as $\nabla r(x)$ is updated at each iteration). By observing that $ \nabla l(x^\top z_i) = l'(x^\top z_i) z_i$, where $l'$ represents the derivative of $l$, it is apparent that we need to store only the scalars $l'(x^\top z_i)$ for $\regsaga$. This reduces the storage $O(n)$ instead of $O(nd)$ in $\saga$.


\section{Linear convergence rates for gradient dominated functions}
Until now the only assumption we used was Lipschitz continuity of gradients. An immediate question is whether the IFO complexity can be further improved under stronger assumptions. We provide an affirmative answer to this question by showing that for gradient dominated functions, a variant of $\saga$ attains linear convergence rate. Recall that a function $f$ is called $\tau$-gradient dominated if around an optimal point $x^*$, $f$ satisfies the following growth condition: 
\begin{equation*}
  f(x) - f(x^*) \leq \tau \|\nabla f(x)\|^2,\quad\forall x\in \reals^d.
\end{equation*}
For such functions, we use the variant of $\saga$ shown in Algorithm~\ref{alg:gd-saga}. Observe that Algorithm~\ref{alg:gd-saga} uses $\saga$ as a subroutine. Alternatively, one can rewrite Algorithm~\ref{alg:gd-saga} in the form of $KT$ iterations of Algorithm~\ref{alg:saga} where one updates $\{\alpha_i\}$ after every $T$ iterations and then selects a random iterate amongst the last $T$ iterates. 
For this variant of $\saga$, we can prove the following linear convergence result.

\begin{algorithm}[tb]\small
   \caption{GD-SAGA$\left(x^0, K, T, \eta \right)$}
   \label{alg:gd-saga}
\begin{algorithmic}
   \STATE {\bfseries Input:} $x^0 \in \mathbb{R}^d$, $K$,  epoch length $m$, step sizes $\eta > 0$
   \FOR{$k=0$ to $K$}
   \STATE $x^{k} = \text{SAGA}(x^{k-1},T,\eta)$
   \ENDFOR
   \STATE {\bfseries Output:} $x^K$
\end{algorithmic}
\end{algorithm}

\begin{theorem}
  \label{thm:gd-saga-thm-nabla}
  If $f$ is $\tau$-gradient dominated, then with $\eta = 1/(3Ln^{2/3})$ and $T = \lceil 24L\tau n^{2/3} \rceil$, the iterates of Algorithm~\ref{alg:gd-saga} satisfy
  $\mathbb{E}[\|f(x^k)\|^2] \leq 2^{-k} \|f(x^0)\|^2,$
  where $x^*$ is an optimal solution of~\eqref{eq:1}.
\end{theorem}
\begin{proof}
  The iterates of Algorithm~\ref{alg:gd-saga} satisfy the bound
\begin{align}
\label{eq:gd-saga-thm-eq}
\mathbb{E}[\| \nabla f(x^k) \|^2] \leq \frac{\mathbb{E}[f(x^{k-1}) - f(x^*)]}{2\tau},
\end{align}
which holds due to Theorem~\ref{thm:nonconvex-gen} given the choices of $\eta$ and $T$ assumed in the statement. However, $f$ is $\tau$-gradient dominated, so $\mathbb{E}[\| \nabla f(x^{k-1}) \|^2] \geq  \mathbb{E}[f(x^{k-1}) - f(x^*)]/\tau$, which combined with~\eqref{eq:gd-saga-thm-eq} concludes the proof.
\end{proof}

An immediate consequence of this theorem is the following.
\begin{corollary}
  \label{cor:saga-gd-oracle}
  If $f$ is $\tau$-gradient dominated, the IFO complexity of Algorithm~\ref{alg:gd-saga} (with parameters from Theorem~\ref{thm:gd-saga-thm-nabla}) to compute an $\epsilon$-accurate solution is $O((n + \tau n^{2/3}) \log(1/\epsilon))$.
\end{corollary}

While we state the result in terms of $\|\nabla f(x)\|^2$, it is not hard to see that for gradient dominated functions a similar result holds for the convergence criterion being $[f(x)- f(x^*)]$.

\begin{theorem}
  \label{thm:gd-saga-thm}
  If $f$ is $\tau$-gradient dominated, with $\eta = 1/(3Ln^{2/3})$ and $T = \lceil 24L\tau n^{2/3} \rceil$, the iterates of Algorithm~\ref{alg:gd-saga} satisfy
  $$\mathbb{E}[f(x^k) - f(x^*)] \leq 2^{-k}[f(x^0) - f(x^*)],$$
  where $x^*$ is an optimal solution to~\eqref{eq:1}.
\end{theorem}

A noteworthy aspect of the above result is the linear convergence rate to a \emph{global} optimum. Therefore, the above result is stronger than Theorem~\ref{thm:nonconvex-gen}. Note that throughout our analysis of gradient dominated functions, no assumptions other than Lipschitz smoothness are placed on the \emph{individual} set of functions $f_i$. We emphasize here that these results can be further improved with additional assumptions (e.g., strong convexity) on the individual functions $f_i$ and on $f$. Also note that $\gd$ can achieve linear convergence rate for gradient dominated functions \cite{Polyak1963}. However, the IFO complexity of $\gd$ is $O(\tau n \log(1/\epsilon))$, which is strictly worse than IFO complexity of GD-$\saga$ (see Corollary~\ref{cor:saga-gd-oracle}).

\section{Minibatch Variant} 
\label{sec:minibatch}
A common variant of incremental methods is to sample a set of indices $I_t$ instead of single index $i_t$ when approximating the gradient. Such a variant is generally referred to as a ``minibatch'' version of the algorithm. Minibatch variants are of great practical significance since they reduce the variance of incremental methods and promote parallelism. Algorithm~\ref{alg:minibatch-saga} lists the pseudocode for a minibatch variant of $\saga$. Algorithm~\ref{alg:minibatch-saga} uses a set $I_t$ of size $|I_t| = b$ for calculating the update $v^t$ instead of a single index $i_t$ used in Algorithm~\ref{alg:saga}. By using a larger $b$, one can reduce the variance due to the stochasticity in the algorithm. Such a procedure is also beneficial in parallel settings since the calculation of the update $v^t$ can be parallelized. For this algorithm, we can prove the following convergence result.

\begin{theorem}
\label{thm:nonconvex-minibatch}
Suppose $f \in \Fc_n$. Let $\eta = b/(3Ln^{2/3})$ and $\beta = L/n^{1/3}$. Then for the output $x_a$ of Algorithm~\ref{alg:minibatch-saga} (with $b < n^{2/3}$) we have $\gamma_n \geq \frac{b}{12Ln^{2/3}}$ and
\begin{align*}
\mathbb{E}[\|\nabla f(x_a)\|^2] &\leq \frac{12Ln^{2/3} [f(x^{0}) - f(x^*)]}{bT},
\end{align*} 
where $x^*$ is an optimal solution to~\eqref{eq:1}.
\end{theorem}
We omit the details of the proof since it is similar to the proof of Theorem~\ref{thm:nonconvex-gen}. Note that the key difference in comparison to Theorem~\ref{alg:saga} is that we can now use a more aggressive step size $\eta = b/(3Ln^{2/3})$ due to a larger minibatch size $b$.  An interesting aspect of the result is the $O(1/b)$ dependence on the minibatch size $b$. As long as this size is not large ($b < n^{2/3}$), one can significantly improve the convergence rate to a stationary point. A restatement of aforementioned result in terms of  IFO complexity is provided below.

\begin{corollary}
\label{cor:saga-nonconvex-oracle-minibatch}
If $f \in \Fc_n$, then the IFO complexity of Algorithm~\ref{alg:minibatch-saga} (with parameters from Theorem~\ref{thm:nonconvex-minibatch} and minibatch size $b <  n^{2/3}$) to obtain an $\epsilon$-accurate solution is $O(n + n^{2/3}/\epsilon)$.
\end{corollary}  

\begin{algorithm}[tb]\small
   \caption{Minibatch-SAGA$\left(x^0,b,T,\eta\right)$}
   \label{alg:minibatch-saga}
\begin{algorithmic}[1]
   \STATE {\bfseries Input:} $x^0 \in \mathbb{R}^d$, $\alpha_{i}^0 = x^0$ for $i \in [n]$,  minibatch size $b$, number of iterations $T$, step size $\eta > 0$
   \STATE $g^{0} = \frac{1}{n} \sum_{i=1}^n \nabla f_{i}(\alpha_i^{0})$
   \FOR{$t=0$ {\bfseries to} $T-1$}
   \STATE Uniformly randomly pick (with replacement) indices sets $I_t, J_t$ of size $b$ from $[n]$ 
   \STATE $v^t =  \frac{1}{|I_t|} \sum_{i \in I_t} (\nabla f_{i}(x^t) - \nabla f_{i}(\alpha_{i_t}^{t})) + g^{t}$
   \STATE $x^{t+1} = x^{t} - \eta v^t$
   \STATE $\alpha_{j}^{t+1} = x^t$ for $j \in J_t$ and $\alpha_{j}^{t+1} = \alpha_{j}^{t}$ for $j \notin J_t$
   \STATE $g^{t+1} = g^t - \frac{1}{n} \sum_{j \in J_t}(\nabla f_{j}(\alpha_{j}^t) - \nabla f_{j}(\alpha_{j}^{t+1}))$
   \ENDFOR
   \STATE {\bfseries Output:} Iterate $x_a$ chosen uniformly random from $\{x^t\}_{t=0}^{T-1}$.
\end{algorithmic}
\end{algorithm}

By comparing the above result with Corollary~\ref{cor:saga-nonconvex-oracle}, we can see that the IFO complexity of minibatch-\saga is the same $\saga$. However, since the $b$ gradients can be computed in parallel, one can achieve (theoretical) $b$ times speedup in multicore and distributed settings. In contrast, the performance $\sgd$ degrades with minibatch size $b$ since the improvement in convergence rate for $\sgd$ is typically $O(1/\sqrt{b})$ but $b$ IFO calls are required at each iteration of minibatch-$\sgd$. Thus, $\saga$ has a much more efficient minibatch version in comparison to $\sgd$.

\subsection*{Discussion of Convergence Rates}
Before ending our discussion on convergence rates, it is important to compare and contrast different convergence rates obtained in the paper. For general smooth nonconvex problems, we observed that $\saga$ has a low IFO complexity of $O(n + n^{2/3}/\epsilon)$ in comparison to $\sgd$ ($O(1/\epsilon^2)$) and $\gd$ ($O(n/\epsilon)$). This difference in the convergence is especially significant if one requires a medium to high accuracy solution, i.e., $\epsilon$ is small. 

Furthermore, for gradient dominated functions, where $\sgd$ obtains a sublinear convergence rate of $O(1/\epsilon^2)$\footnote{For $\sgd$, we are not aware of any better convergence rates for gradient dominated functions.} as opposed to fast linear convergence rate of a variant of \saga (see Theorem~\ref{cor:saga-gd-oracle}). It is an interesting future work to explore other setups where we can achieve stronger convergence rates. 

From our analysis of minibatch-\saga in Section~\ref{sec:minibatch}, we observe that $\saga$ profits from mini-batching much more than $\sgd$. In particular, one can achieve a (theoretical) linear speedup using mini-batching in $\saga$ in parallel settings. On the other hand, the performance of $\sgd$ typically degrades with minibatching. In summary, $\saga$ enjoys all the benefits of $\gd$ like constant step size, efficient minibatching with much weaker dependence on $n$.

Notably, $\saga$, unlike $\sgd$, does not use any additional assumption of bounded gradients (see Theorem~\ref{thm:sgd-oracle} and Corollary~\ref{cor:saga-nonconvex-oracle}). Moreover, if one uses a constant step size for $\sgd$, we need to have an advance knowledge of the total number of iterations $T$ in order to obtain the convergence rate mentioned in Theorem~\ref{thm:sgd-oracle}.

\section{Experiments}

We present our empirical results in this section. For our experiments, we study the problem of binary classification using nonconvex regularizers. The input consists of tuples $\{(z_i, y_i)\}_{i=1}^n$ where $z_i \in \mathbb{R}^d$ (commonly referred to as features) and $y_i \in \{-1,1\}$ (class labels). We are interested in the empirical loss minimization setup described in Section~\ref{sec:reg}. Recall that problem of finite sum with regularization takes the form
\begin{equation}
\label{eq:gen-linear-model}
  \min_{x\in \reals^d}\ f(x) := \frac{1}{n}\sum_{i=1}^n f_i(x) + r(x).
\end{equation}
For our experiments, we use logistic function for $f_i$, i.e., $f_i(x) = \log(1+\exp(-y_i x^\top z_i))$ for all $i \in [n]$. All  $z_i$ are normalized such that $\|z_i\| = 1$. We observe that the loss function has Lipschitz continuous gradients.  The function $r(x) = \lambda \sum_{i=1}^d \alpha x_i^2/(1+\alpha x_i^2)$ is chosen as the regularizer (see \cite{Antoniadis2009}). Note that the regularizer is nonconvex and smooth. In our experiments, we use the parameter settings of $\lambda =0.001$ and $\alpha = 1$ for all the datasets.

\begin{figure*}[!t]
\centering
   \begin{minipage}[b]{.24\textwidth}
   \includegraphics[width=\textwidth]{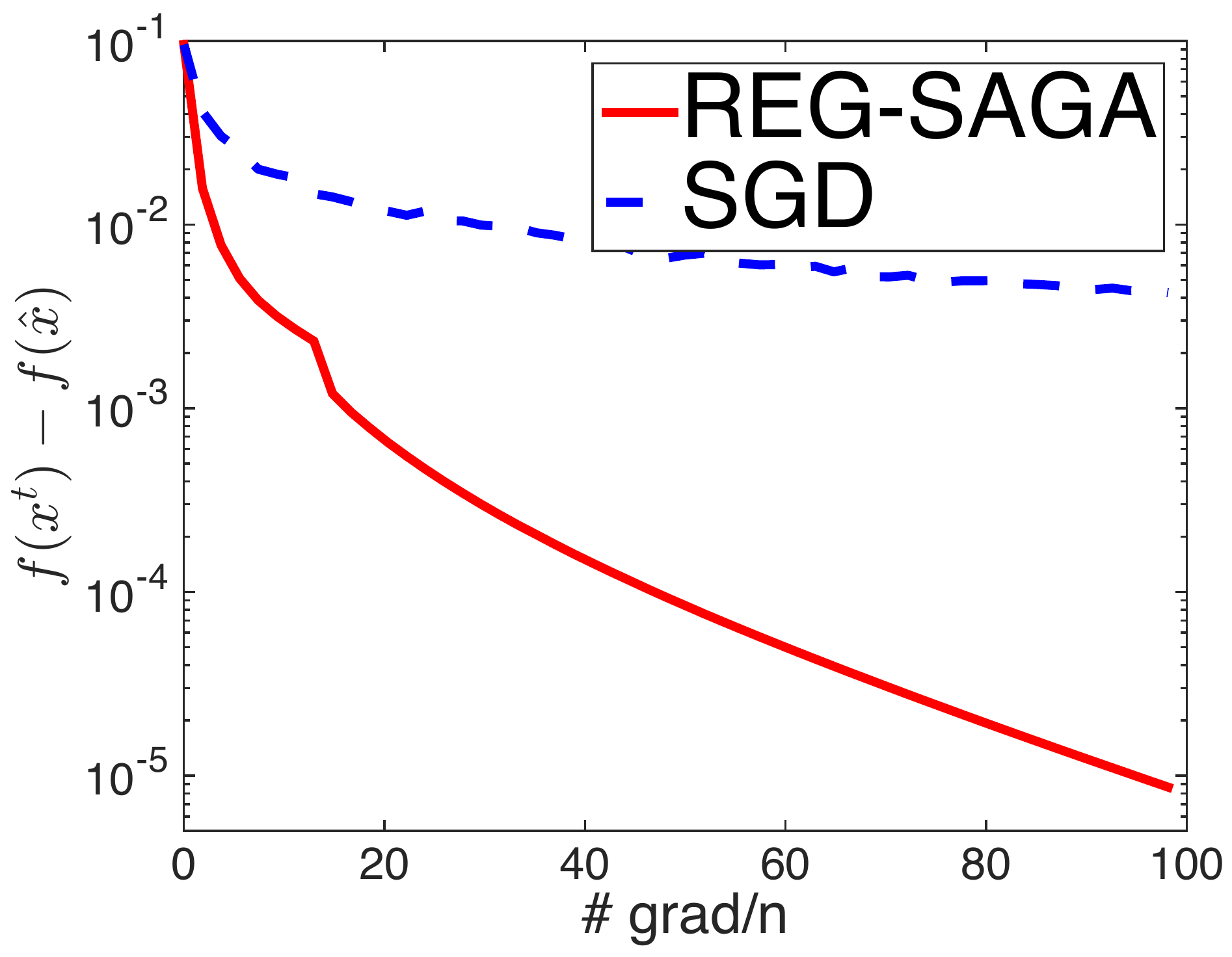}
   \end{minipage} %
   \begin{minipage}[b]{.24\textwidth}
   \includegraphics[width=\textwidth]{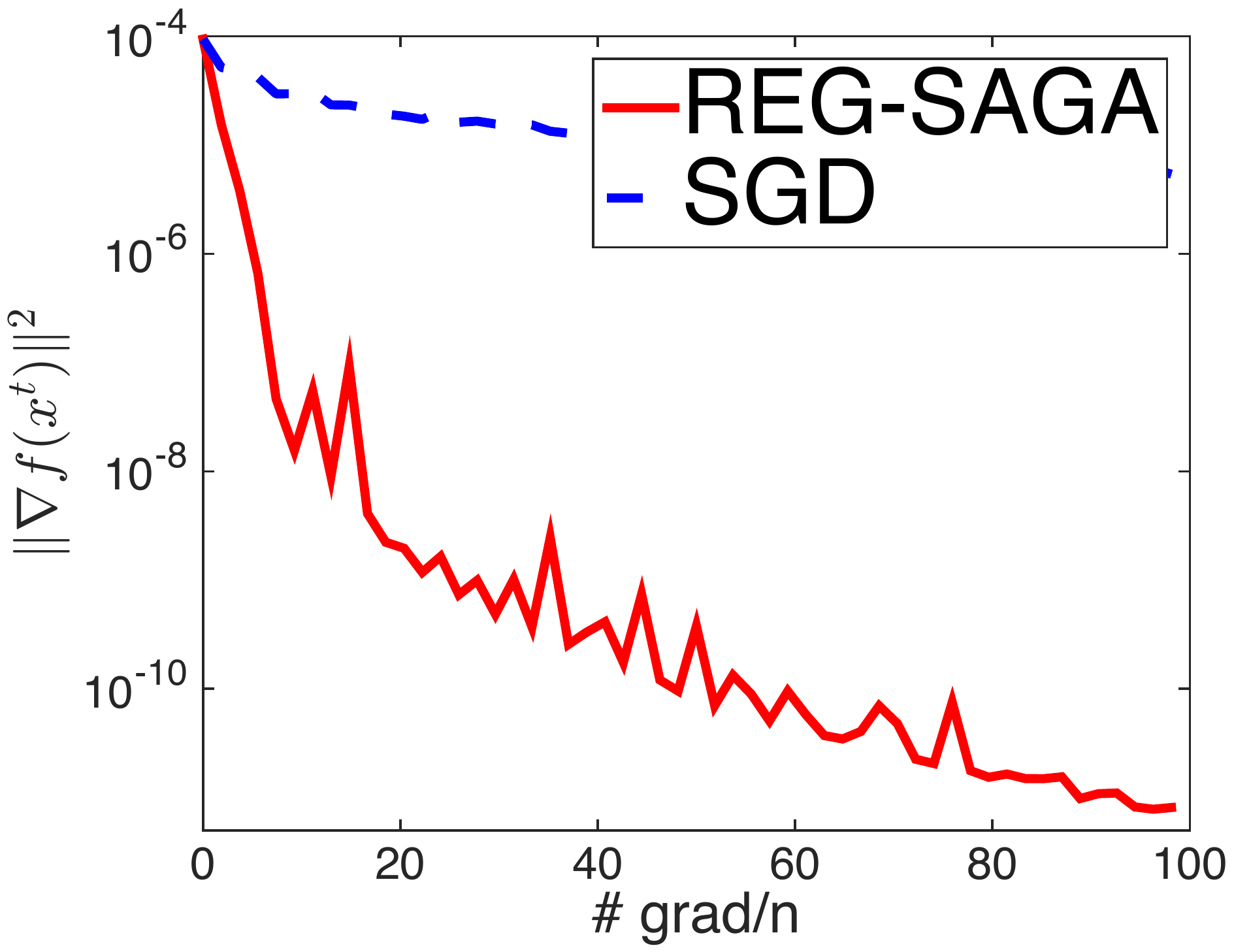}
   \end{minipage}
   \begin{minipage}[b]{.24\textwidth}
   \includegraphics[width=\textwidth]{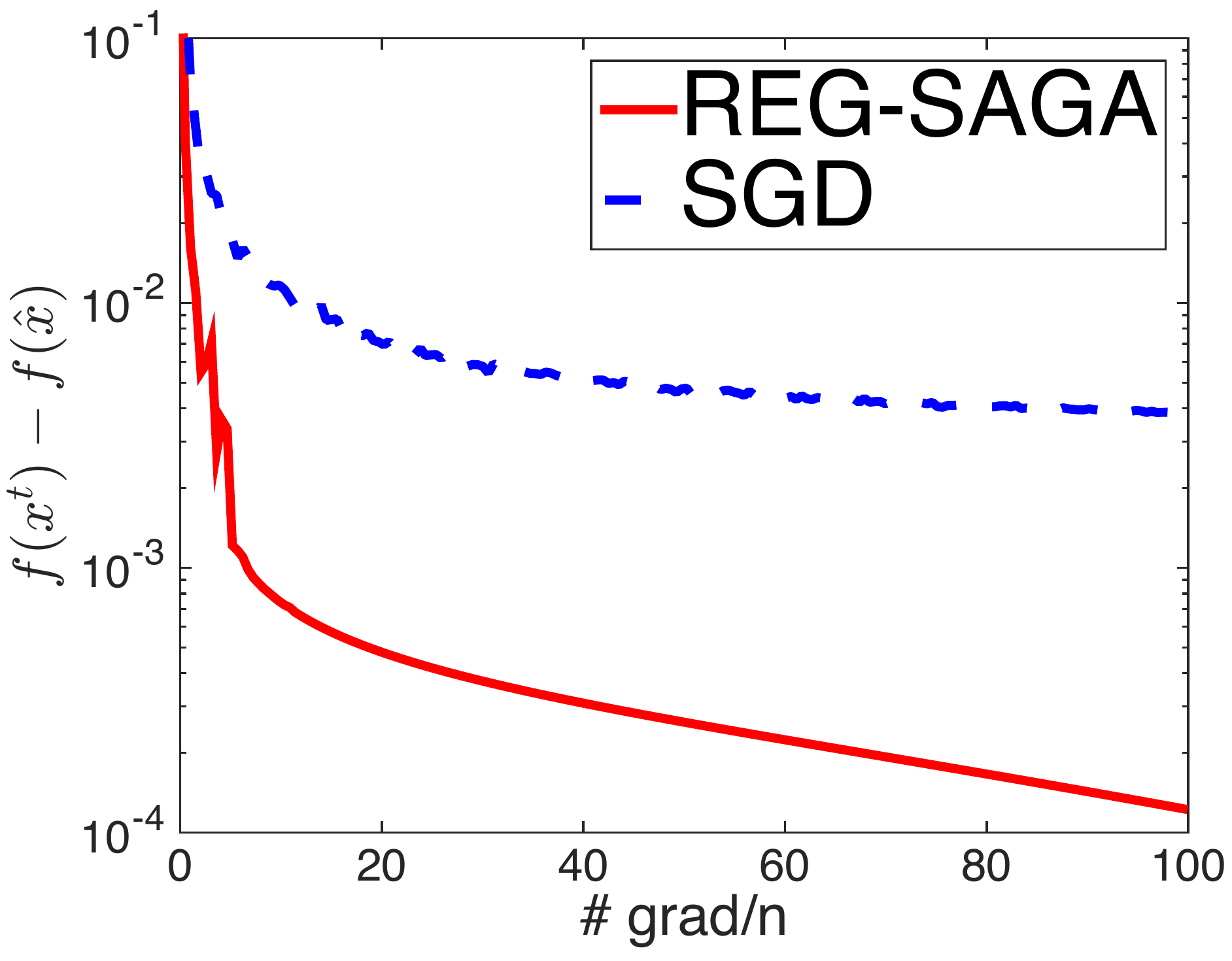}
   \end{minipage}
   \begin{minipage}[b]{.24\textwidth}
   \includegraphics[width=\textwidth]{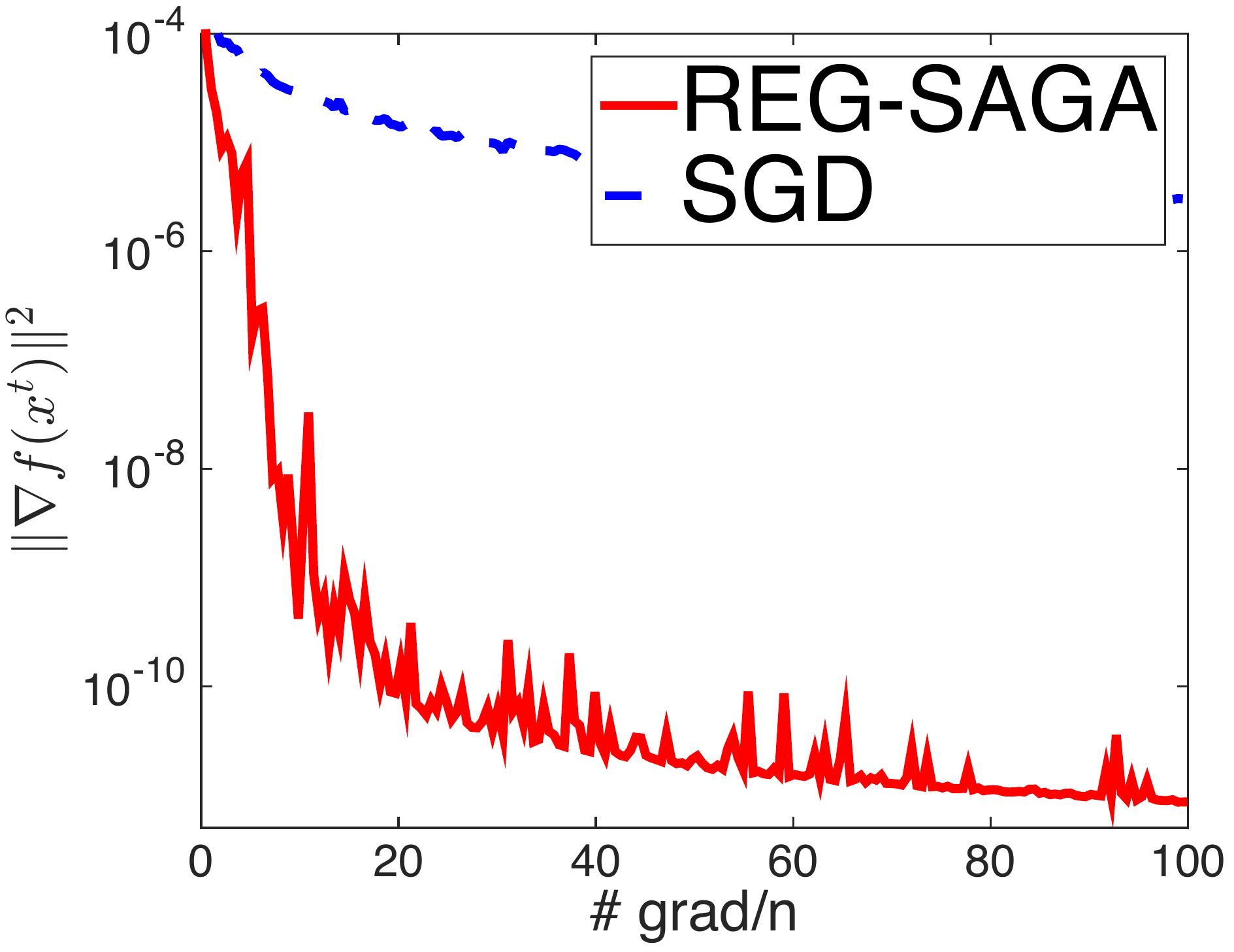}
   \end{minipage} %
	\caption{Results for nonconvex regularized generalized linear models (see Equation~\eqref{eq:gen-linear-model}). The first and last two figures correspond to rcv1 and realsim datasets respectively. The results compare the performance of $\regsaga$ and $\sgd$ algorithms. Here $\hat{x}$ corresponds to the solution obtained by running $\gd$ for a very long time and using multiple restarts. As seen in the figure, $\regsaga$ converges much faster than $\sgd$ in terms of objective function value and the stationarity gap $\|\nabla f(x)\|^2$.}
	\label{fig:results}
\end{figure*}

We compare the performance of $\sgd$ (the \emph{de facto} incremental method for nonconvex optimization) with nonconvex $\regsaga$ in our experiments. The comparison is based on the following criteria: (i) the objective function value (also called training loss in this context), which is the main goal of the paper; and (ii) the stationarity gap $\| \nabla f(x)\|^2$, the criteria used for our theoretical analysis. For the step size of $\sgd$, we use the popular $t-$inverse schedule $\eta_t = \eta_0(1 + \eta'\lfloor t/n \rfloor)^{-1}$, where $\eta_0$ and $\eta'$ are tuned so that \sgd gives the best performance on the training loss. In our experiments, we also use $\eta' = 0$; this results in a fixed step size for $\sgd$. For $\regsaga$, a fixed step size is chosen (as suggested by our analysis) so that it gives the best performance on the objective function value, i.e., training loss. Note that due to the structure of the problem in~\eqref{eq:gen-linear-model}, as discussed in Section~\ref{sec:reg}, the storage cost of $\regsaga$ is just $O(n)$. 

Initialization is critical to many of the incremental methods like $\regsaga$. This is due to the stronger dependence of the convergence on the initial point (see Theorem~\ref{thm:nonconvex-gen}). Furthermore, one has to obtain $\nabla f_i(\alpha^0_i)$ for all $i \in [n]$ in $\regsaga$ algorithm (see Algorithm~\ref{alg:saga}). For initialization of both $\sgd$ and $\regsaga$, we make one pass (without replacement) through the dataset and perform the updates of $\sgd$ during this pass. Doing so not only allows us to also obtain a good initial point $x^0$ but also to compute the initial values of $\nabla f(\alpha_i^0)$ for $i \in [n]$. Note that this choice results in a variant of $\regsaga$ where $\alpha_i^0$ are different for various $i$ (unlike the pseudocode in Algorithm~\ref{alg:saga}). The convergence rates of this variant can be shown to be similar to that of Algorithm~\ref{alg:saga}.

Figure~\ref{fig:results} shows the results of our experiments. The results are on two standard UCI datasets, `rcv1' and `realsim'\footnote{The datasets can be downloaded from \url{https://www.csie.ntu.edu.tw/~cjlin/libsvmtools/datasets/binary.html}.}. The plots compare the criteria mentioned earlier against the number of IFO calls made by the corresponding algorithm. For the objective function, we look at the difference between the objective function ($f(x^t)$) and the best objective function value obtained by running $\gd$ for a very large number of iterations using multiple initializations (denoted by $f(\hat{x})$). As shown in the figure, $\regsaga$ converges much faster than $\sgd$ in terms of objective value. Furthermore, as supported by the theory, the stationarity gap for $\regsaga$ is very small in comparison to $\sgd$. Also, in our experiments, the selection of step size was much easier for $\regsaga$ when compared to $\sgd$. Overall the empirical results for nonconvex regularized problems are promising. It will be interesting to apply the approach for other smooth nonconvex problems. 

\section{Conclusion}

In this paper, we investigated a fast incremental method ($\saga$) for nonconvex optimization. Our main contribution in this paper to show that $\saga$ can provably perform better than both  $\sgd$ and $\gd$ in the context of nonconvex optimization. We also showed that with additional assumptions on function $f$ in~\eqref{eq:1} like gradient dominance, $\saga$ has linear convergence to the \emph{global} minimum as opposed to sublinear convergence of $\sgd$. Furthermore, for large scale parallel settings, we proposed a minibatch variant of $\saga$ with stronger theoretical convergence rates than $\sgd$ by attaining linear speedups in the size of the minibatches. One of the biggest advantages of $\saga$ is the ability to use fixed step size. Such a property is important in practice since selection of step size (learning rate) for $\sgd$ is typically difficult and is one of its biggest drawbacks. 

\bibliographystyle{custom}
\bibliography{bibfile}

\section*{Appendix}

\section{Proof of Lemma~\ref{lem:nonconvex-saga}}
\label{sec:proofs}
\begin{proof}
Since $f$ is $L$-smooth, from Lemma~\ref{lem:descent-lemma}, we have
\begin{align*}
&\mathbb{E}[f(x^{t+1})] \leq \mathbb{E}[f(x^{t}) + \langle \nabla f(x^t), x^{t+1} - x^t \rangle \nonumber \\
& \qquad \qquad \qquad \qquad \qquad + \tfrac{L}{2} \| x^{t+1} - x^t \|^2].
\end{align*}
We first note that the update in Algorithm~\ref{alg:saga} is unbiased i.e., $\mathbb{E}[v^t] = \nabla f(x^t)$. By using this property of the update on the right hand side of the inequality above, we get the following:
\begin{align}
&\mathbb{E}[f(x^{t+1})] \leq  \mathbb{E}[f(x^{t}) - \eta_t \|\nabla f(x^{t})\|^2 + \tfrac{L\eta^2}{2} \|v^t \|^2].
\label{eq:saga-proof-eq1}
\end{align}
Here we used the fact that $x^{t+1} - x^{t} = -\eta v^t$ (see Algorithm~\ref{alg:saga}). Consider now the Lyapunov function
$$
R^{t} := \mathbb{E}[f(x^{t}) + \tfrac{c_{t}}{n} \sum_{i=1}^n \|x^{t} - \alpha_{i}^t\|^2].
$$
For bounding $R^{t+1}$ we need the following:
\begin{align}
\label{eq:aux-term}
& \frac{1}{n} \sum_{i=1}^n \mathbb{E}[\|x^{t+1} - \alpha_{i}^{t+1}\|^2] \nonumber \\
&= \frac{1}{n} \sum_{i=1}^n \left[ \frac{1}{n} \mathbb{E}\|x^{t+1} - x^{t}\|^2 + \frac{n-1}{n} \underbrace{\mathbb{E}\|x^{t+1} - \alpha_{i}^{t}\|^2}_{T_1} \right].
\end{align}
The above equality from the definition of $\alpha^{t+1}_i$ and the uniform randomness of index $j_t$ in Algorithm~\ref{alg:saga}. The term $T_1$ in~\eqref{eq:aux-term} can be bounded as follows
\begin{align}
&T_1 = \mathbb{E}[\|x^{t+1} - x^t + x^t - \alpha_{i}^{t}\|^2] \nonumber \\
&= \mathbb{E}[\|x^{t+1} - x^t\|^2 + \|x^t - \alpha_{i}^{t}\|^2 \nonumber + 2\langle x^{t+1} - x^t, x^t - \alpha_{i}^{t}\rangle] \nonumber \\
&= \mathbb{E}[\|x^{t+1} - x^t\|^2 + \|x^t - \alpha_{i}^{t}\|^2] - 2\eta \mathbb{E}[\langle \nabla f(x^t), x^t -\alpha_{i}^{t}\rangle] \nonumber \\
&\leq \mathbb{E}[\|x^{t+1} - x^t\|^2 + \|x^t - \alpha_{i}^{t}\|^2] \nonumber \\
& \qquad \qquad + 2 \eta \mathbb{E}\left[\tfrac{1}{2\beta} \|\nabla f(x^t)\|^2 + \tfrac{1}{2}\beta \| x^t - \alpha_{i}^{t}\|^2 \right].
\label{eq:saga-proof-eq2}
\end{align}
The second equality again follows from the unbiasedness of the update of $\saga$. The last inequality follows from a simple application of Cauchy-Schwarz and Young's inequality. Plugging~\eqref{eq:saga-proof-eq1} and~\eqref{eq:saga-proof-eq2} into $R^{t+1}$, we obtain the following bound:
\begin{align}
& R^{t+1} \leq \mathbb{E}[f(x^{t}) - \eta \|\nabla f(x^{t})\|^2 + \tfrac{L\eta^2}{2} \|v^t \|^2] \nonumber \\
&  + \mathbb{E}[c_{t+1}\|x^{t+1} - x^t\|^2 + c_{t+1}\frac{n-1}{n^2}  \sum_{i=1}^n \|x^t - \alpha_i^t\|^2] \nonumber \\
&  + \frac{2(n-1)c_{t+1}\eta}{n^2} \sum_{i=1}^n \mathbb{E}\left[\tfrac{1}{2\beta} \|\nabla f(x^t)\|^2 + \tfrac{1}{2}\beta \| x^t - \alpha_i^t \|^2 \right] \nonumber\\
&\leq \mathbb{E}[f(x^{t}) - \left(\eta - \tfrac{c_{t+1}\eta}{\beta}\right) \|\nabla f(x^{t})\|^2 \nonumber\\
&  \qquad + \left(\tfrac{L\eta^2}{2} + c_{t+1}\eta^2 \right)\mathbb{E}[\|v^t\|^2] \nonumber\\
&  \qquad + \left(\frac{n-1}{n} c_{t+1} + c_{t+1}\eta\beta \right) \frac{1}{n} \sum_{i=1}^n \mathbb{E}\left[\| x^t - \alpha_i^t\|^2 \right].
\label{eq:saga-proof-eq3}
\end{align}
To further bound the quantity in~\eqref{eq:saga-proof-eq3}, we use Lemma~\ref{lem:nonconvex-variance-lemma} to bound $\mathbb{E}[\|v^{t}\|^2]$, so that upon substituting it into ~\eqref{eq:saga-proof-eq3}, we obtain
\begin{align*}
& R^{t+1} \leq \mathbb{E}[f(x^{t})] \nonumber \\
& - \left(\eta - \tfrac{c_{t+1}\eta}{\beta} - \eta^2L - 2c_{t+1}\eta^2\right) \mathbb{E}[\|\nabla f(x^{t})\|^2] \nonumber\\
& + \left[c_{t+1}\bigl(1 - \tfrac{1}{n} + \eta\beta + 2\eta^2L^2\bigr)+\eta^2L^3\right]
 \tfrac{1}{n} \sum_{i=1}^n \mathbb{E}\left[\| x^t - \alpha_i^t \|^2 \right] \nonumber \\
& \leq R^{t} - \bigl(\eta - \tfrac{c_{t+1}\eta}{\beta} - \eta^2L - 2c_{t+1}\eta^2\bigr) \mathbb{E}[\|\nabla f(x^{t})\|^2].
\end{align*}
The second inequality follows from the definition of $c_{t}$ i.e., $c_{t} = c_{t+1}(1 - \tfrac{1}{n} + \eta\beta + 2\eta^2L^2) +  \eta^2L^3$ and $R^t$ specified in the statement, thus concluding the proof.
\end{proof}

\section{Other Lemmas}

The following lemma provides a bound on the variance of the update used in $\saga$ algorithm. More specifically, it bounds the quantity $\mathbb{E}[\|v^t\|^2]$. A more general result for bounding the variance of the minibatch scheme in Algorithm~\ref{alg:minibatch-saga} can be proved along similar lines.

\begin{lemma}
\label{lem:nonconvex-variance-lemma}
Let $v^t$ be computed by Algorithm~\ref{alg:saga}. Then,
\begin{align*}
\mathbb{E}[\|v^t\|^2] \leq 2\mathbb{E}[\|\nabla f(x^{t})\|^2] + \frac{2L^2}{n} \sum_{i=1}^n \mathbb{E}[\|x^{t} - \alpha_i^t\|^2].
\end{align*}
\end{lemma}
\begin{proof}
For ease of exposition, we use the notation
$$
\zeta^{t} = \left(\nabla f_{i_t}(x^{t}) - \nabla f_{i_t}(\alpha_{i_t}^{t})\right)
$$
Using the convexity of $\|\!\cdot\!\|^2$ and the definition of $v^t$ we get
\begin{align*}
&\mathbb{E}[\|v^t\|^2] = \mathbb{E}[\|\zeta^{t} + \tfrac{1}{n}\sum_{i=1}^n\nabla f(\alpha_i^t) \|^2] \\
&= \mathbb{E}[\| \zeta^{t}  + \tfrac{1}{n}\sum_{i=1}^n\nabla f(\alpha_i^t) - \nabla f(x^{t}) +  \nabla f(x^{t}) \|^2]\\
&\leq 2\mathbb{E}[\|\nabla f(x^{t})\|^2] + 2 \mathbb{E}[\|\zeta^{t} - \mathbb{E}[\zeta^{t}]\|^2]\\
&\leq 2\mathbb{E}[\|\nabla f(x^{t})\|^2] + 2 \mathbb{E}[\|\zeta^{t}\|^2].
\end{align*}
The first inequality follows from the fact that $\|a + b\|^2 \leq 2(\|a\|^2 + \|b\|^2)$ and that $ \mathbb{E}[\zeta^{t}] = \nabla f(x^{t}) - \tfrac{1}{n} \sum_{i=1}^n \nabla f(\alpha_i^{t})$. The second inequality is obtained by noting that for a random variable $\zeta$, $\mathbb{E}[\|\zeta - \mathbb{E}[\zeta]\|^2] \leq \mathbb{E}[\|\zeta\|^2]$. Using Jensen's inequality in the inequality above, we get
\begin{align*}
&\mathbb{E}[\|v^t\|^2] \\
& \leq 2\mathbb{E}[\|\nabla f(x^{t})\|^2] + \frac{2}{n} \sum_{i=1}^n \mathbb{E}[\|\nabla f_{i_t}(x^{t}) - \nabla f_{i_t}(\alpha_i^{t})\|^2] \\
& \leq 2\mathbb{E}[\|\nabla f(x^{t})\|^2] + \frac{2L^2}{n} \sum_{i=1}^n \mathbb{E}[\|x^{t} - \alpha_i^{t}\|^2].
\end{align*}
The last inequality follows from $L$-smoothness of $f_{i_t}$, thus concluding the proof.
\end{proof}

The following result provides a bound on the function value of functions with Lipschitz continuous gradients.
\begin{lemma}
\label{lem:descent-lemma}
Suppose the function $f:\mathbb{R}^d \rightarrow \mathbb{R}$ is $L$-smooth, then the following holds
\begin{align*}
f(x)  \leq f(y) + \langle \nabla f(y), x - y \rangle + \frac{L}{2} \| x - y \|^2,
\end{align*}
for all $x, y \in \mathbb{R}^d$.
\end{lemma}

\end{document}